\documentclass{amsart}

   \usepackage[
    top    = 2.0cm,
    bottom = 1.95cm,
    left   = 2cm,
    right  = 2cm]    {geometry}

\usepackage{amssymb}
\usepackage{stmaryrd}
\usepackage{fullpage}
\usepackage[colorlinks]{hyperref}

\newtheorem{theorem}{Theorem}[section]
\newtheorem{lemma}[theorem]{Lemma}
\newtheorem{cor}[theorem]{Corollary}

\theoremstyle{definition}

\theoremstyle{remark}

\newtheorem{tab}[theorem]{\bf Table}

\numberwithin{equation}{section}

\newcommand{\FF}{{\mathbb{F}}}
\newcommand{\KK}{{\mathbb{K}}}

\newcommand{\bC}{{\mathbf{C}}}
\newcommand{\bF}{{\mathbf{F}}}

\newcommand{\bZ}{{\mathbf{Z}}}

\newcommand{\Aut}{{\operatorname{Aut}}}

\newcommand{\SL}{{\operatorname{SL}}}
\newcommand{\Sp}{{\operatorname{Sp}}}

\newcommand{\GL}{{\operatorname{GL}}}
\newcommand{\Gal}{{\operatorname{Gal}}}

\newcommand{\SCRSp}{{\operatorname{SCRSp}}}
\newcommand{\CHA}{{\operatorname{char}}}

\newcommand{\Div}{{\operatorname{Div}}}

\newcommand{\SP}{{\operatorname{SP}}}

\newcommand{\EP}{{\operatorname{EP}}}
\newcommand{\NEP}{{\operatorname{NEP}}}

\newcommand{\PC}{{\operatorname{PC}}}
\newcommand{\NPC}{{\operatorname{NPC}}}

\newcommand{\GF}{\mbox{GF}}

\let\nor=\triangleleft

\begin{document}

\title{Regular orbits of finite primitive solvable groups, III}

\author{YONG YANG}
\address{Department of Mathematics, Texas State University, 601 University Drive, San Marcos, TX 78666, USA.}

\makeatletter
\email{yang@txstate.edu}
\makeatother

\author{ALEXEY VASIL'EV}
\address{Sobolev Institute of mathematics, Novosibirsk State University, 630090, Novosibirsk}

\makeatletter
\email{a.vasilev1@g.nsu.ru}

\author{EVGENY VDOVIN}
\address{Sobolev Institute of mathematics, Novosibirsk State University, 630090, Novosibirsk}

\makeatletter
\email{vdovin@math.nsc.ru}

\Large
\date{}




\begin{abstract}

    Suppose that a finite solvable group $G$ acts faithfully, irreducibly and quasi-primitively on a finite vector space $V$, and $G$ is not metacyclic.  Then $G$ always has a regular orbit on $V$ except for a few  ``small" cases.

\end{abstract}

\maketitle
\Large
\section{Introduction} \label{sec:introduction}

Let $G$ be a finite group and $V$ a finite, faithful and completely reducible $G$-module. It is a classical theme to study orbit structure of $G$ acting on $V$. One of the most important and natural questions about orbit structure is to establish the existence of an orbit of a certain size. For a long time, there has been a deep interest and need to examine the size of the largest possible orbits in linear group actions. The orbit $\{v^g \ |\ g \in G\}$ is called regular, if $\bC_G(v)=1$ holds or equivalently the size of the orbit $v^G$ is $|G|$. It is well known that the existence of regular orbits has been studied extensively in the literature with many applications to some important questions of character theory and conjugacy classes of finite groups.

In ~\cite{PalfyPyber}, P\'alfy and Pyber asked if it is possible to classify all pairs $A$, $G$ with $(|A|, |G|)=1$ such that $A \leq \Aut(G)$ has a regular orbit on $G$. While the task is pretty challenging, at least for primitive solvable linear groups, we can say something along this line.

Suppose that a finite solvable group $G$ acts faithfully, irreducibly and quasi-primitively on a finite vector space $V$ of dimension $n$ over a finite field of order $q$ and characteristic $p$. Then $G$ has a uniquely determined normal subgroup $E$ which is a direct product of extraspecial $p$-groups for various $p$. We denote $e=\sqrt{|E/\bZ(E)|}$ (an invariant measuring the complexity of the group). It is proved in \cite[Theorem 3.1]{YY2} and \cite[Theorem 3.1]{YY3} that if $e=5,6,7$ or $e \geq 10$ and $e \neq 16$, then $G$ always have regular orbits on $V$.

If $e=1$, then $G\leq \Gamma(q^n)$, and it is possible that $G=\Gamma(q^n)$, while $\Gamma(q^n)$ does not have a regular orbit for $n\geq 2$. So for $e=1$ one cannot expect that $G$ necessarily possesses a regular orbit. In this case $G$ is metacyclic and thus there are infinitely many metacyclic primitive linear groups that do not have regular orbits.

There are also other examples for $e>1$, when $G$ does not possess a regular orbit. However, the main result of the paper implies the following.

\begin{theorem}\label{Main}
    Let $G$ be a solvable group, acting faithfully, irreducibly, and quasi-primitively on a finite vector space $V$ and $\vert V\vert>5^{18}$. Assume also that $G$ is not metacyclic. Then $G$ possesses a regular orbit on $V$.
\end{theorem}

Note that we know only a few examples of maximal irreducible primitive solvable subgroups of $\GL(V)$ that are not metacyclic and does not possess a regular orbit. We provide a much narrower list of possible groups without regular orbit in Table~\ref{exception} (see Corollary \ref{except}). 

The information on the existence of a regular orbit  has been used by several authors to study a variety of problems in the field (for example ~\cite{DOLFIEmanuele1,DOLFIEmanuele2,LNW,Ponomarenko,TurullWolf,YY4,YY5}), and the results in this paper can be used to simplify some of the proofs in ~\cite{DOLFIEmanuele1,DOLFIEmanuele2,TurullWolf,YY4,YY6}.


The bound we obtain here is not perfect. However we shall point out that only finite number of cases are left, and hopefully they can be studied by using computations on computer. We hope that in the near future, by relying on the developments of the computer programs like~\cite{GAP}, one might be able to obtain a complete classification for these remaining cases. 

Another application of the results obtained in the paper, is to study the bases of solvable permutation groups and $k$-closures of solvable groups. Usually one may   reduce the question  to the investigation of primitive permutation groups.
So consider  a primitive solvable permutation group $G\leq S_{p^k}$. Then $G=A \rtimes L$, where $A$ is elementary abelian of order $p^k$ and $L$ is an irreducible solvable subgroup of $\GL(k,p)$. If $L$ is not primitive as a subgroup of $\GL(k,p)$, then $G$ can be written as a wreath product of proper subgroups with product action, and this case can be reduced to smaller cases. Thus the primary interest is the case, when $G$ is a primitive permutation group, and $L$ is a primitive linear group. If $L$ has a regular orbit on $A$, then the base size of $G$ is at most $2$, and $G$ is equal to its $3$-closure.

\section{Notation and Lemmas} \label{sec:Notation and Lemmas}

Notation:
\begin{enumerate}
    \item  Let $G$ be a finite group, let $S$ be a subset of $G$ and let $\pi$ be a set of different primes. For each prime $s$, we denote $\SP_s(S)=\{\langle x \rangle \ | \ o(x)=s, x \in S \}$ and $\EP_s(S)=\{x \ |\ o(x)=s, x \in S \}$. We denote $\SP(S)=\bigcup \SP_s(S)$, $\EP(S)=\bigcup \EP_s(S)$ and $\EP_{\pi}(S)=\bigcup_{s \in \pi}\EP_s(S)$. We denote $\NEP(S)=|\EP(S)|$, $\NEP_s(S)=|\EP_s(S)|$ and $\NEP_{\pi}(S)=|\EP_{\pi}(S)|$.

    \item Let $n$ be an even integer, $q$ a power of a prime. Let $V$ be a standard symplectic vector space of dimension $n$ of $\FF_q$. We use $\SCRSp(n,q)$ or $\SCRSp(V)$ to denote the set of all solvable subgroups of $\Sp(V)$ which acts completely reducibly on $V$.

    \item Let $V$ be a finite vector space and let $G \leq \GL(V)$. We define $\PC(G,V,s,i)=\{x \ |\ x \in \EP_s(G)$ and $\dim(\bC_V(x))=i\}$ and $\NPC(G,V,s,i)=|\PC(G,V,s,i)|$. We will drop $V$ in the notation when it is clear in the context.

    \item If $V$ is a finite vector space of dimension $n$ over $\GF(q)$, where $q$ is a prime power, we denote by $\Gamma(q^n)=\Gamma(V)$ the semilinear group of $V$, i.e.,
\[\Gamma(q^n)=\{x \mapsto ax^{\sigma}\ |\ x \in \GF(q^n), a \in \GF(q^n)^{\times}, \sigma \in \Gal(\GF(q^n)/\GF(q))\}.\]

    \item We use $H \wr S$ to denote the wreath product of $H$ with $S$ where $H$ is a group and $S$ is a permutation group.
    \item We use $\bF(G)$ to denote the Fitting subgroup of $G$. 
    \item Let $n$ be a positive integer, and we use $\Div(n)$ to denote the number of different prime divisors of $n$. Clearly if $n=1$, $\Div(n)=0$.

\end{enumerate}

\begin{theorem} \label{Strofprimitive}

    Suppose that a finite solvable group $G$ acts faithfully, irreducibly and
    quasi-primitively on an $d$-dimensional finite vector space $V$ over a finite field $\FF$ of characteristic $p$.
    Then every normal abelian subgroup of $G$ is cyclic and $G$ has normal subgroups $Z \leq U \leq F \leq A \leq G$ and a characteristic subgroup $E \leq F$ such that,
    \begin{enumerate}
        \item $F=EU$ is a central product where $Z=E \cap U=\bZ(E)$ and $\bC_G(F) \leq F$;
        \item $F/U \cong E/Z$ is a direct sum of completely reducible $G/F$-modules;
        \item There is decomposition $E = E_1 \times E_2 \times \ldots \times E_k$, $E_i$ is an extraspecial $r_i$-group for $i=1,\dots,s$ for some distinct primes $r_i$, $|E_i|=r_i^{2n_i + 1}$ for some $n_i \geq 1$. Denoting $e_i = r_i^{n_i}$, we have $e=e_1 \dots e_s$ divides $d$ and $\gcd(p,e)=1$;
        \item $A=\bC_G(U)$ and $G/A \lesssim \Aut(U)$, $A/F$ acts faithfully on $E/Z$;
        \item $A/\bC_A(E_i/Z_i) \lesssim \Sp(2n_i,r_i)$;
        \item $U$ is cyclic and acts fixed point freely on $W$ where $W$ is an irreducible submodule of $V_U$;
        \item $|V|=|W|^{eb}$ for some integer $b$;
        \item $G/A$ is cyclic and $|G:A| \mid \dim(W)$. $G=A$ when $e=d$;
        \item Let $g \in G \backslash A$, assume that $o(g)=t$ where $t$ is a prime and let $|W|=p^a$. Then $t \mid a$ and we can view the action of $g$ on $U$ as follows, $U \leq \FF_{p^{a}}^*$ and $g \in \Gal(\FF_{p^{a}}:\FF_p)$.
    \end{enumerate}
\end{theorem}
\begin{proof}
    This is slightly reformulated \cite[Theorem 2.2]{YY3}.
\end{proof}

In this paper we investigate only the case when $e$ is a prime power, so Theorem~\ref{Strofprimitive} can be simplified to the following

\begin{theorem} \label{Strofprimitiveprime}

    Suppose that a finite solvable group $G$ acts faithfully, irreducibly and
    quasi-primitively on a $d$-dimensional finite vector space $V$ over a finite field $\FF$ of characteristic $p$.
    Then every normal abelian subgroup of $G$ is cyclic and $G$ has normal subgroups $Z \leq U \leq F \leq A \leq G$ and $E \leq F$ as in the Theorem~\ref{Strofprimitive}. Suppose $|F:U|=e^2$ is a prime power. Then the following holds:
    \begin{enumerate}
        \item $U$ is cyclic, $|U|= p^a - 1$ for some $a \ge 1$ and $W$ can be identified with the span of $U$ which is isomorphic to $\GF(p^a)$;
        \item $\bC_G(F) \le F = \bF(G)$ and $e$ divides $p^a - 1$;
        \item $E$ is an extraspecial $r$-group for a prime $r$, $|E / Z|=e^2=r^{2n}$. Furthermore $e$ divides $d$ and $\gcd(p,e)=1$;
        \item $A=\bC_G(U)$ and $G/A \lesssim \Aut(U)$, $A/F$ acts faithfully on $E/Z$;
        \item $A/\bC_A(E/Z) \lesssim \Sp(2n,r)$;
        \item $G/A$ is cyclic and $|G:A| \mid \dim(W)$. $G=A$ when $e=d$;
    \end{enumerate}
\end{theorem}


\renewcommand\theenumi{(\arabic{enumi})}%
\renewcommand*\labelenumi{\theenumi}

The following Lemma~\ref{decomptriv} is just a linear algebra fact.
\begin{lemma}\label{decomptriv}
    Suppose that a finite group $G$ acts faithfully and irreducibly on a finite vector space $V$ over field $\FF$ of order $q$. Let $s$ be prime divisor of $q-1$ and $\lambda_1,\lambda_2,\ldots,\lambda_s$ are all the $s$-roots of the unity in the field $\FF_{q}$. Let $x \in \EP_s(G), z \in \EP_s(Z(G))$.
    \begin{enumerate}
        \item There exists $1 \le k \le s$ such that $vz = \lambda_k v$ for every $v\in V$.
        \item There is a unique decomposition $V = V_{\lambda_1} \oplus \ldots \oplus V_{\lambda_s}$ of $V$ in the direct sum of $x$-submodules such that $vx = \lambda_i v$ for every $v \in V_{\lambda_i}$. In particular, if $\lambda_1 = 1$, then $V_{\lambda_1} = \bC_V(x)$.
        \item Since $s$ is a prime, we derive $\{ \lambda_k^j \mid j=0,1,\ldots,s-1\} = \{\lambda_1, \lambda_2,\ldots,\lambda_s\}$ and
              \[
                  V = \bC_V(x) \oplus \bC_V(xz) \oplus \ldots \oplus \bC_V(x z^{s-1})
                  .\]
    \end{enumerate}
\end{lemma}

\begin{lemma}  \label{roughestimate}
    Suppose that a finite solvable group $G$ acts faithfully and quasi-primitively on a finite vector space $V$ over the field $\FF$. Let $g \in \EP_s(G)$ and we use the notation in Theorem ~\ref{Strofprimitiveprime}.
    \begin{enumerate}
        \item If $g \in F$ then $|\bC_V(g)| \leq |W|^{\frac 1 s eb}$.
        \item If $g \in A \backslash F$ then $|\bC_V(g)| \leq |W|^{\lfloor \frac 3 4 e \rfloor b}$.
        \item If $g \in A \backslash F$, $s\geq 3$ and $s \nmid |E|$, then $|\bC_V(g)| \leq |W|^{\lfloor \frac 1 2 e \rfloor b}$.
        \item If $g \in A \backslash F$, $s=2$ and $s \nmid |E|$, then $|\bC_V(g)| \leq |W|^{\lfloor \frac 2 3 e \rfloor b}$.
        \item If $g \in G \backslash A$ then $|\bC_V(g)| \leq |W|^{\frac 1 s eb}$.
    \end{enumerate}
\end{lemma}
\begin{proof}
    (1) is a slight improvement of \cite[Lemma 2.4(1)]{YY3}. We only need to consider the cases when $s \geq 3$. Assume first that $\langle g \rangle \leq F$, $\langle g \rangle \not \leq U$. Then $s \neq r$ and $\langle g \rangle \leq E \nor G$, with $E$ an extraspecial $s$-group. We consider the restriction $V_E$. It is a direct sum of faithful irreducible $E$-modules, since $\bZ(E) \leq U$ acts fixed point freely on $V$. As dimensions of centralizers do not change by extending the ground field and they add up in direct sums, we can assume that $V$ is an irreducible, faithful $E$-module on an algebraically closed field $\KK$. If $\chi$ is the (Brauer) character corresponding to $V$ , then, as $r=\CHA(\KK) \neq s$, $\dim_{\KK}(\bC_V(S)) = [\chi_S, 1_S] = \frac 1 s \dim_{\KK}(V)$ since $\chi(x) = 0$ for every $x \in E \backslash  Z(E)$, and (1) is proved.

    (2), (3), (4) and (5) follow from \cite[Lemma 2.4]{YY3}.

\end{proof}

\begin{lemma}  \label{mygood}
    Suppose that a finite solvable group $G$ acts faithfully, irreducibly and quasi-primitively on a finite vector space $V$ over a field $\FF$. Using the notation in Theorem ~\ref{Strofprimitiveprime}, let $x \in \EP_s(A \backslash F)$ and $(s,\CHA \FF)=1$. Let $C/Z=\bC_{E/Z}(x)$, we call $x$ a good element if $[x,C]=1$, and we call $x$ a bad element if it is not good.
    \begin{enumerate}
        \item Assume $x$ is a bad element and let $\beta=e/s$, then $|\bC_V(x)| \leq |W|^{\beta b}$.

        \item Assume $x$ is a good element and $|\bC_{E/Z}(x)| \leq a$, let \[\beta = \lfloor \frac 1 s (e+(s-1)a^{1/2}) \rfloor \]
              then $|\bC_V(x)| \leq |W|^{\beta b}$.

    \end{enumerate}
\end{lemma}
\begin{proof}
    This is \cite[Lemma 2.6]{YY2}.
\end{proof}

\begin{lemma}  \label{basiccounting}
    Assume $G$ satisfies Theorem ~\ref{Strofprimitiveprime} and we adopt the notation in it. Let $s$ be a prime and $x \in \EP_s(A \backslash F)$ and assume $|\bC_{E/Z}(x)|=r^m$. Define $U_{s}= \gcd(|U|, s)$. We have the following,
    \begin{enumerate}
        \item $\NEP_s(A \backslash F) \leq \NEP_s(A/F)|F|$.
        \item $\NEP_s(xF) \leq M \cdot U_{s}$ where
              \[M=\left\{ \begin{array}{lll} r^{2n} & \mbox{if $s=r \neq 2$}; \\ r^{2n-m} & \mbox{if $s \neq r$}; \\ 2^m & \mbox{if $s=r=2$}. \end{array} \right.\]
        \item Assume that $s=2$ and $x$ is a good element. Define $S=\{y \ |\ y \in \EP_2(xF)$ and $y$ is a good element$\}$, then $|S| \leq M \cdot U_{2}$ where
              \[M=\left\{ \begin{array}{lll} r^{2n-m} & \mbox{if $r \neq 2$}; \\ 2^m & \mbox{if $r=2$ and $n \geq m$}; \\ 2^{2n-m} & \mbox{if $r=2$ and $n<m$}. \end{array} \right.\]
    \end{enumerate}
\end{lemma}
\begin{proof}
    (1)  follow from ~\cite[Lemma 2.7]{YY2}. (2) and (3) are slight improvements of \cite[Lemma 2.7]{YY2}. We only show the proof of (2) here. By the proof of ~\cite[Lemma 2.7(3)]{YY2}, we know that $\NEP_s(x F/U) \leq M$. Let $\alpha \in A$ and $o(\alpha)=s$, we consider $\NEP_s(\alpha U)$. Since $U \leq \bZ(A)$, $\NEP_s(\alpha U) \leq U_{s}$ and the result follows. (3) can be proved similarly.
\end{proof}

\begin{lemma}  \label{counttop2u}
    Let $G$ be a finite solvable group and $U$ a normal cyclic subgroup of $G$. Let $\alpha \in G \backslash U$ and $o(\alpha)=s$. Assume we can view the action of $\alpha$ on $U$ as follows, $U \leq \FF_{q^{sn}}^*$ and $\alpha \in \Gal(\FF_{q^{sn}}:\FF_q)$. Then $\NEP_s(\alpha U) \leq \frac {q^{sn}-1} {q^n-1}$.
\end{lemma}
\begin{proof}
    Let $u \in U$ and assume $o(\alpha u)=s$, then $\alpha u \alpha u \dots  \alpha u=1$.

    This implies that $\alpha u \alpha^{-1} \alpha^{2} u \alpha^{-2} \dots  u \alpha^{s-1} \alpha^{s} u=1$ and $u^{q^{n}} \cdot u^{q^{2n}} \cdots u^{q^{(s-1)n}} \cdot u=u^{\frac {q^{sn}-1} {q^n-1}}=1$.
\end{proof}

\begin{lemma}  \label{preresult}
    We have $|\SL(2,3)|=24$, $\NEP_2(\SL(2,3))=1$ and $\NEP_3(\SL(2,3))=8$.
\end{lemma}
\begin{proof}
    It is easy to check.
\end{proof}

\begin{lemma}  \label{BdSCRSp}
    Let $n$ be an even integer and $V$ be a standard symplectic vector space of dimension $n$ of field  $\FF$. Let $G \in \SCRSp(n,\FF)$.
    \begin{enumerate}
        \item Let $(n,\FF)=(2,\FF_2)$, then $G  \lesssim S_3$, $|G| \leq 6$, $\NEP_2(G) \leq 3$, $\NEP_3(G) \leq 2$, $\NPC(G,2,1) \leq 3$.
        \item Let $(n,\FF)=(4,\FF_2)$, then $|G| \leq 6^2 \cdot 2$, $\NEP_2(G) \leq 21$, $\NEP_3(G) \leq 8$, $\NEP_5(G) \leq 4$, $\NEP_{\{2,3,5\}'} = 0$, $\NEP(G) \leq 29$, $\NPC(G,2,3) \leq 6$, $\NPC(G,2,2) \leq 15$, $\NPC(G,2,1)=0$, $\NPC(G,3,2) \leq 4$.
        \item Let $(n,\FF)=(6,\FF_2)$, then $|G| \leq 6^4$, $\NEP_2(G) \leq 135$, $\NEP_3(G)\leq 242$, $\NEP_{\{2,3\}'}(G) \leq 6$, $\NPC(G,2,5) \leq 9$, $\NPC(G,2,4) \leq 45$, $\NPC(G,2,3) \leq 108$, $\NPC(G,2,2)=0$, $\NPC(G,2,1)=0$.
        \item Let $(n,\FF)=(8,\FF_2)$, then $|G| \leq 6^4 \cdot 24$, $\NEP(G)\leq 1883$, $\NPC(G,2,7) \leq 12$, $\NPC(G,2,6) \leq 90$, $\NPC(G,2,5) \leq 324$, $\NPC(G,2,4) \leq 513$, $\NPC(G,2,3)=0$, $\NPC(G,2,2)=0$, $\NPC(G,2,1)=0$.
        \item Let $(n,\FF)=(2,\FF_3)$, then $G \lesssim \SL(2,3)$ and $|G| \leq 24$. $\NEP_2(G) \leq 1$ and $\NEP_3(G) \leq 8$.
        \item Let $(n,\FF)=(4,\FF_3)$, then $|G| \leq 24^2 \cdot 2$ and $\NEP(G) \leq 107$, $\NEP_2(G) \leq 95$, $\NEP_3(G) \leq 95$, $\NEP_5(G) \leq 64$ and $G$ will have no elements with other prime order. Assume $G \not\lesssim \SL(2,3) \wr S_2$, then $\NPC(G,3,3)=0$.
    \end{enumerate}
    This follows from \cite[Lemma 2.17]{YY2}.
\end{lemma}

\section{Main Theorem} \label{sec:maintheorem}
\begin{theorem} \label{maintheorem}
    Assume that a finite solvable group $G$ acts faithfully, irreducibly and quasi-primitively on a finite vector space $V$. By Theorem ~\ref{Strofprimitive}, $G$ will have a uniquely determined normal subgroup $E$ which is a direct product of extraspecial $r$-groups for various $r$ and $e=\sqrt{|E/\bZ(E)|}$. For every $e \in \{2,3,4,8,9,16\}$ suppose that at least one of the following holds:
    \begin{enumerate}
        \item $|W|$ is prime and $|W|$ is not less than the number in the first row of Table~\ref{maintable};
        \item $|W|$ is a power of prime and $|W|$ is not less than the number in the second row of Table~\ref{maintable};
        \item $b > k, k=1,2,3,4$ and $|W|$ is not less than the number in the row number $2+k$ of Table~\ref{maintable};
    \end{enumerate}
    Then $G$ has at least one regular orbit on $V$.
\end{theorem}
\begin{tab}\label{maintable}{Lower bounds for $|W|$}
    \centering
    \begin{tabular}{|c|c|c|c|c|c|c|}
        \hline
                           & $e=16$ & $e=9$ & $e=8$ & $e=4$ & $e=3$ & $e=2$ \\ \hline
        $|W|$ is prime     & 7      & 31    & 23    & 79    & 31    & 31    \\ \hline
        $|W|$ is arbitrary & 7      & 31    & 27    & 281   & 439   & 5693  \\ \hline
        $b > 1$            & 3      & 7     & 7     & 13    & 13    & 19    \\ \hline
        $b > 2$            & 3      & 4     & 5     & 5     & 7     & 7     \\ \hline
        $b > 3$            & 3      & 4     & 3     & 5     & 4     & 5     \\ \hline
        $b > 4$            & 3      & 4     & 3     & 3     & 4     & 3     \\ \hline
    \end{tabular}

\end{tab}

\begin{proof}

    In order to show that $G$ has at least one regular orbit on $V$ it suffices to check that \[\left | \bigcup_{P \in \SP(G)}\bC_V(P) \right | < |V|.\] In the following cases we will divide the set $\SP(G)$ into a union of sets $A_i$. Clearly $$\left | \bigcup_{P \in \SP(G)}\bC_V(P) \right | \leq \sum_i \left | \bigcup_{P \in A_i} \bC_V(P) \right |.$$

    We will find $\beta_i < e$ such that
    \begin{equation}\label{ineqforbeta}
        |\bC_V(P)| \leq |W|^{\beta_i b}
    \end{equation}
    for all $P \in A_i$.
    We will find $a_i$ such that $|A_i| \leq a_i$.
    Since $|V|=|W|^{eb}$ it suffices to check that \[\sum_i a_i \cdot (|W|^{\beta_i b}-1) < |W|^{eb}-1. \] We call this inequality $\star$.\\ 

    Let $e=16$.
    Thus in the notations of Theorem~\ref{Strofprimitiveprime} $|F / U| = |E / Z| = 2^8,\ r=2,\ 2n = 8,\ 2 \mid |W|-1$ and $A/F \in \SCRSp(8,2)$. By Lemma ~\ref{BdSCRSp}(4), it follows that $|A/F| \leq 6^4 \cdot 24$.

    Define $A_i$-s in the following way:
    \begin{enumerate}
        \item $A_1=\{\langle x \rangle\ |\ x \in \EP_2(F \backslash U)\}$,
        \item $A_2=\{\langle x \rangle\ |\ x \in \EP_2(A \backslash F)$, $x$ is good and $\bC_{\bar E}(x)=2^6 \}$,
        \item $A_3=\{\langle x \rangle\ |\ x \in \EP_2(A \backslash F)$, $x$ is good and $\bC_{\bar E}(x)=2^4 \}$, $A_3'=\{\langle x \rangle\ |\ x \in \EP_2(A \backslash F)$, $x$ is good and $\bC_{\bar E}(x)=2^2 \}$,
        \item $A_4 = \{ \langle x \rangle  \mid\ x \in \EP_2(A \backslash F)$ and $x$ is bad$\}$,
        \item $A_5=\{\langle x \rangle\ |\ x \in \EP_s(A \backslash F)$ for all primes $s \geq 3\}$,
        \item $A_6 = \left\{ \langle x \rangle  \mid x \in \EP_2(G \backslash A) \right\}$,
        \item $A_{7,s} = \left\{ \langle x \rangle  \mid x \in \EP_s(G \backslash A)\right\}$ for $s \ge 3$ and $A_7 = \bigcup\limits_{s \ge 3}  A_{7,s}$.
    \end{enumerate}

    For all $\langle x \rangle \in A_1$ we have $|\bC_V(x)|\le |W|^{8b}$ by Lemma~\ref{roughestimate}(1).
    Thus we set $\beta_1=8$ (see \ref{ineqforbeta}).
    Analogously,
    \begin{itemize}
        \item for all $\langle x \rangle \in A_2$ we have $|\bC_V(x)|\le |W|^{12b}$ by Lemma~\ref{mygood}(2),
        \item for all $\langle x \rangle \in A_3$ we have $|\bC_V(x)|\le |W|^{10b}$ by Lemma~\ref{mygood}(2),
        \item for all $\langle x \rangle \in A_4$ we have $|\bC_V(x)|\le |W|^{8b}$ by Lemma~\ref{mygood}(1),
        \item for all $\langle x \rangle \in A_5$ we have $|\bC_V(x)|\le |W|^{8b}$ by Lemma~\ref{roughestimate}(3),
        \item for all $\langle x \rangle \in A_6$ we have $|\bC_V(x)|\le |W|^{8b}$ by Lemma~\ref{roughestimate}(5),
        \item for all $\langle x \rangle \in A_{7,s}$ for prime $s\ge 3$ we have $|\bC_V(x)|\le |W|^{\frac{16}{s}b}$ by Lemma~\ref{roughestimate}(5).
    \end{itemize}
    Thus we set $\beta_2=12$, $\beta_3=10$, $\beta_4=8$, $\beta_5 = 8$, $\beta_6 = 8$, $\beta_{7,s} = \frac{16}{s}$.

    To estimate $|A_1|$, we estimate the number of involutions in $F$. An involution in $F$ is either an involution in $E$ or an element $ec=e^3c^3$, where $e \in E, c \in U$ and $|e|=|c|=4$. Note that there are at most two $c \in U$ of order 4 since $U$ is cyclic by Theorem~\ref{Strofprimitiveprime}. There are at most 271 involutions and at most $240$ elements of order $4$ in $E$, thus $|A_1| \le  271 +\frac{240\cdot 2}{2}-1=510$ and we set $a_1 = 510$.

    Let $\langle x \rangle \in A_2$. Then $xF \in \PC(A / F, E / Z, 2, 6)$. By Lemma~\ref{BdSCRSp}(4), it follows that $\NPC(A / F, E / Z, 2, 6) \le 90$. By Lemma~\ref{basiccounting}(3), it follows that $\NEP_2(xF) \cap A_2 \le 2^2 \cdot 2$. Thus we set $a_2 = 90 \cdot 2^2 \cdot 2$.

    As in the previous case, for $\langle x \rangle \in A_3$ we have $xF \in PC(A / F, E / Z, 2, 4)$, $NPC(A / F, E / Z, 2, 4) \le 513$ by Lemma~\ref{BdSCRSp}(4) and $\NEP_2(xF) \cap A_3 \le 2^4 \cdot 2$ by Lemma~\ref{basiccounting}(3), thus we set $a_3 = 513 \cdot 2^4 \cdot 2$.

    For $\langle x \rangle \in A_3'$ we have $xF \in \PC(A / F, E / Z, 2, 2)$ and since $NPC(A / F, E / Z, 2, 2) =0$ by Lemma~\ref{BdSCRSp}(4), it follows that $A_3' = \varnothing$.

    Let $\langle x \rangle \in A_4$. Then $xF \in \EP_2(A / F)$. By Lemma~\ref{BdSCRSp}(4), it follows that $\NEP_2(A / F) \le 12 + 90 + 324 + 513 = 939$. Obviously, $|\bC_{E / Z} (x)| \le |E / Z| = 2^8$. By Lemma~\ref{basiccounting}(2), it follows that $\NEP_2(xF) \le 2^8\cdot 2$. Thus we set $a_4 = 939 \cdot 2^8 \cdot 2$.

    Let $\langle x \rangle \in A_5,\ o(x) = s$. As in the previous case, we have $xF \in \EP(A / F),\ \NEP(A / F) \le 1883$ by Lemma~\ref{BdSCRSp}(4) and $\NEP_s(xF) \le 2^8 \cdot s$ by Lemma~\ref{basiccounting}(2). Since the largest prime dividing the order of $\Sp_8(2)$ is 17, we set $a_5 = 1883\cdot 2^8 \cdot 17$.

    Let $\langle x \rangle \in A_6$. Denote the natural homomorphism from $G$ to $G / U = \overline{G}$ by $\overline{\phantom{x}}$. Then $\overline{x} = xU \in \NEP_2(\overline{G} \backslash \overline{A})$ and $\overline{x} \overline{A} \in \NEP_2(\overline{G} / \overline{A})$. By Theorem ~\ref{Strofprimitiveprime}, it follows that there is at most one subgroup of order 2 in the group $G / A \simeq \overline{G} / \overline{A}$, thus
    \[
        |A_6| \le 1 \cdot \NEP_2(\overline{x}\overline{A}) \cdot \NEP_2(xU) \le |\overline{A}| \cdot \NEP_2(xU) \le |A / F| \cdot |F / U| \cdot \NEP_2(xU)
        .\]

    By Lemma~\ref{counttop2u}, it follows that $\NEP_2(xU) \le |W|^{\frac{1}{2}} + 1$, thus we set $a_6 = 6^4 \cdot 24 \cdot 2^8 \cdot (|W|^{\frac{1}{2}} + 1)$.

    By analogy to the previous case, we set $a_{7,s} = 6^4 \cdot 24 \cdot 2^8 \cdot \frac{|W|-1}{|W|^{\frac{1}{s}-1}}$.

    To estimate the sum of summands in $\star$ corresponding to subgroups in $A_7$, first, note the following relations (the last one holds for $s \ge 3$)
    \begin{gather*}
        \frac{|W| - 1}{|W|^{\frac{1}{s}} - 1} (|W|^{\frac{16b}{s}}-1) = \frac{|W|^{\frac{16b}{s}}-1}{ |W|^{\frac{1}{s}}-1} (|W|-1) =
        \\
        = (1 + |W|^{\frac{1}{s}} + \ldots + |W|^{\frac{16b-1}{s}})(|W|-1) \\
        \le 16b |W|^{\frac{16b-1}{s}}(|W|-1)
        \le 16b |W|^{\frac{16b-1}{3}}(|W|-1).
    \end{gather*}
    Second, note that $s$ divides $|G : A|$, which divides $\dim (W)$ by Theorem~\ref{Strofprimitiveprime}(6). This implies that there are at most $\Div(\dim (W))$ possibilities for $s$.
    Thus
    \[
        \sum\limits_{s \ge 3} a_{7,s} (|W|^{\beta_{7,s}} - 1) \le \Div (\dim W) \cdot 6^4 \cdot 24 \cdot 2^8 \cdot 16b \cdot |W|^{\frac{16b-1}{3}} (|W| - 1)
        .\]

    \bigskip

    In the case $|W|$ is prime, sets $A_{7,s}$ and $A_6$ are empty and we set $a_6 = a_{7,s} = 0$.

    It is routine to check that $\star$ is satisfied when $|W| \geq 7$ or $|W|\geq 3$ and $b \ge  2$. \\

    Let $e=9$.
    Thus in the notations of Theorem~\ref{Strofprimitiveprime} $|F / U| = |E / Z| = 3^4, r = 3, 2n = 4, 3 \mid |W|-1$ and $A/F \in \SCRSp(4,3)$. By Lemma~\ref{BdSCRSp}(6), it follows that $|A/F| \leq 24^2 \cdot 2$.

    Define $A_i$-s in the following way:
    \begin{enumerate}
        \item $A_1=\{\langle x \rangle \ |\ x \in \EP_3(F \backslash U) \}$,
        \item $A_2=\{\langle x \rangle \ |\ x \in \EP_2(A \backslash F) \}$,
        \item $A_3=\{\langle x \rangle \ |\ x \in \EP_3(A \backslash F)$ and $|\bC_V(x)| \geq |W|^{6b}\}$,
        \item $A_4=\{\langle x \rangle \ |\ x \in \EP_3(A \backslash F)$ and $|\bC_V(x)| \leq |W|^{5b}\}$,
        \item $A_5 = \{ \langle x \rangle \ |\ x \in \EP_5(A \backslash F) \},$
        \item $A_6=\{\langle x \rangle \ |\ x \in \EP_2(G \backslash A) \}$,
        \item $A_{7,s} = \left\{ \langle x \rangle  \mid x \in \EP_s(G / A) \right\}$ and $A_7 = \bigcup\limits_{s \ge 3} A_{6,s}$.
    \end{enumerate}

    We set $\beta_i$ as follows:
    \begin{enumerate}
        \item $\beta_1 = 3$ (see Lemma~\ref{roughestimate}(1)),
        \item $\beta_2 = 6$ (see Lemma~\ref{roughestimate}(2)),
        \item  Let $\langle x \rangle \in A_3$. Then $x$ is a good element and $\bC_E(x) = 3^3$. Indeed, suppose $x$ is a bad element. Then Lemma~\ref{mygood}(1) implies $|\bC_V(x)|\le |W|^{3b}$ hence $x \not\in A_3$, contradiction. Thus $x$ is a good element. Since $x \not\in F$, it follows that $|\bC_E(x)|\le 3^3$. Suppose that $|\bC_E(x)| \le  3^2$. Then Lemma~\ref{mygood}(2) implies $|\bC_V(x)|\le |W|^{5b}$ hence $x \not\in A_3$, contradiction. Thus $|\bC_E(x)| = 3^3$ and Lemma~\ref{mygood}(2) implies $|\bC_V(x)|\le |W|^{6b}$, thus $|\bC_V(x)| = |W|^{6b}$ and we set $\beta_3 = 6$.
        \item $\beta_4 = 5$ (see definition of $A_4$),
        \item $\beta_5 = 4$ (see Lemma~\ref{roughestimate}(3)),
        \item $\beta_6 = 4.5$ (see Lemma~\ref{roughestimate}(5)),
        \item $\beta_{7,s} = \frac{9}{s}$ (see Lemma~\ref{roughestimate}(5)).
    \end{enumerate}

    Each subgroup in $A_1$ is generated by an element of order 3 of $E$.
    There are at most $242$ elements of order $3$ in $E$, so there are at most $\frac{242}{2}$ subgroups in $A_1$ and we set $a_1 = 121$.

    We have $|A_2| \le  95 \cdot 3^4 \cdot 2$ by Lemma ~\ref{BdSCRSp}(6) and Lemma ~\ref{basiccounting}(2) (see the case $A_4$ for $e=16$). Thus we set $a_2 = 95 \cdot 3^4 \cdot 2$

    Let $x \in \EP_3(A \backslash F)$ s.~t. $\langle x \rangle \in A_3$. As we mentioned before, $\bC_V(x)|=|W|^{6b}$.
    Let $z \in U, |z| = 3$. By Lemma~\ref{decomptriv}, there is a decomposition $V = \bC_V(x) \oplus \bC_V(xz) \oplus \bC_V(xz^2)$, thus for both $y = xz$ and $ y = xz^2$ we have $|\bC_V(y)| \le |W|^{9b-6b}$ hence $y \not\in A_3$.
    Remind that $\bC_E(x) = 3^3$, thus $xF \in \PC(A / F, E / Z,3,3)$.
    By Lemma ~\ref{BdSCRSp}(6) and Lemma ~\ref{preresult}, it follows that $\NPC(A / F, E / Z,3,3) \le 8 \cdot 2=16$.
    By Lemma ~\ref{basiccounting}(2), it follows that $\NEP_3(xF) \le 3^5$.
    Clearly, $xz, xz^2 \in xF$ and $\langle x \rangle = \langle x^2 \rangle$, thus we set $a_3 = \frac{1}{3} \cdot  \frac{1}{2}\cdot 16 \cdot 3^5$.

    We have $|A_4|  \leq 95 \cdot 3^5 /2$ by Lemma ~\ref{BdSCRSp}(6) and Lemma ~\ref{basiccounting}(2) (see the case $A_4$ for $e=16$). Thus we set $a_4 = \leq 95 \cdot 3^5 /2$.

    We have $|A_5| \le 64\cdot 5 \cdot 3^4 / 4$ by Lemma ~\ref{BdSCRSp}(6) and Lemma ~\ref{basiccounting}(2) (see the case $A_4$ for $e=16$). Thus we set  $a_5 = 64 \cdot 5 \cdot 3^4 /4$.

    By analogy with the case $A_6$ for $e = 16$, we have $|A_6| \le  24^2 \cdot 2 \cdot  3^4 \cdot (|W|^{\frac{1}{2}} + 1)$ and $|A_{7,s}| \le  24^2 \cdot 2 \cdot 3^4 \cdot \frac{|W|-1}{|W|^{\frac{1}{s}} - 1}$ by Lemma~\ref{BdSCRSp}(6) and Lemma~\ref{counttop2u}. Thus we set $a_6 = 24^2 \cdot 2 \cdot  3^4 \cdot (|W|^{\frac{1}{2}} + 1)$ and $a_{7,s} = 24^2 \cdot 2 \cdot 3^4 \cdot \frac{|W|-1}{|W|^{\frac{1}{s}} - 1}$. As in the case $e = 16$, we obtain the following estimate for the summands in $\star$ corresponding to subgroups in $A_7$:
    \[
        \sum\limits_{s \ge 3} a_{7,s} (|W|^{\beta_{7,s}} - 1) \le \Div (\dim W) \cdot 9 b \cdot 24^2 \cdot 2 \cdot 3^4 \cdot |W|^{\frac{9b-1}{3}} (|W| - 1)
        .\]

    It is routine to check that $\star$ is satisfied when $|W| \geq 31$.
    Also $\star$ is satisfied if $|W|\geq 7$ and $b \geq 2$ or $|W| \ge 4$ and $b \ge 3$. \\

    Let $e=8$. Thus in the notations of Theorem~\ref{Strofprimitiveprime} $|F / U| = |E / Z| = 2^6,\ r = 2,\ 2n = 6,\ 2 \mid |W|-1$ and $A/F \in \SCRSp(6,2)$. By Lemma~\ref{BdSCRSp}(3), it follows that $|A/F| \leq 6^4$.

    Define $A_i$-s in the following way:
    \begin{enumerate}
        \item $A_1=\{\langle x \rangle\ |\ x \in \EP_2(F \backslash U)\}$,
        \item $A_2=\{\langle x \rangle\ |\ x \in \EP_2(A \backslash F)$, $x$ is a good element and $\bC_{\bar E}(x)=2^4 \}$,
        \item $A_3=\{\langle x \rangle\ |\ x \in \EP_2(A \backslash F)$ and $x$ is bad$\}$,
        \item $A_4=\{\langle x \rangle\ |\ x \in \EP_s(A \backslash F)$ for all primes $s \geq 3\}$,
        \item $A_5=\{\langle x \rangle\ |\ x \in \EP_2(G \backslash A)\}$,
        \item $A_{6,s} = \left\{ \langle x \rangle  \mid x \in \EP_s(G \backslash A),\ s\ge 3 \right\}$ and $A_6 = \bigcup\limits_{s \ge 3} A_{6,s}$.
    \end{enumerate}

    We set $\beta_i$ as follows:
    \begin{enumerate}
        \item $\beta_1 = 4$ (see Lemma~\ref{roughestimate}(1)),
        \item $\beta_2 = 6$ (see Lemma~\ref{mygood}(2)),
        \item $\beta_3 = 4$ (see Lemma~\ref{mygood}(3)),
        \item $\beta_4 = 4$ (see Lemma~\ref{roughestimate}(3)),
        \item $\beta_5 = 4$ (see Lemma~\ref{roughestimate}(5)),
        \item $\beta_{6,s} = \frac{8}{s}$ (see Lemma~\ref{roughestimate}(5)).
    \end{enumerate}

    By analogy to the case $A_1$ for $e=16$, we set $a_1 = 71+56-1 = 126$.

    By analogy to the case $A_3$ for $e=16$, we have $|A_2| \le  45 \cdot 2^3 \cdot 2$ by Lemma ~\ref{BdSCRSp}(3) and Lemma ~\ref{basiccounting}(3).
    Thus we set $a_2 = 45 \cdot 2^3 \cdot 2$.
    Let $\langle x \rangle \in A_2,\ z \in U,\ |z| =2$. Clearly, $\langle xz \rangle$ is in $A_2$ too.
    By Lemma~\ref{decomptriv}, we have $V = \bC_V(x) \oplus \bC_V(xz)$ and the worst possible case is $|\bC_V(x)| = |W|^{6b}$ and $|\bC_V(xz)|=|W|^{2b}$.
    Thus $A_2$ can be split into two sets $A_{21}$ and $A_{22}$, and we set $\beta_{21} = 6, \beta_{22} = 2$ and $a_{21} = a_{22} = \frac{a_2}{2} = 45\cdot 2^3$.

    By analogy to the case $A_4$ for $e=16$ we have $a_4 \le  135\cdot 2^6\cdot 2$ by Lemma ~\ref{BdSCRSp}(3) and Lemma ~\ref{basiccounting}(2), thus we set $a_4 = 135\cdot 2^6\cdot 2$.

    By analogy to the case $A_5$ for $e = 16$, since 7 is the largest prime dividing $|\Sp_6(2)|$, it follows that $|A_4| \le 248 \cdot 2^6 \cdot 7 \cdot \frac{1}{2}$ by Lemma ~\ref{BdSCRSp}(3) and Lemma ~\ref{basiccounting}(2) and we set $a_4 =  248 \cdot 2^6 \cdot 7 \cdot \frac{1}{2}$.

    In the same way as in the case $A_6$ for $e=16$, we have $|A_5| \le 6^4 \cdot 2^6 \cdot (|W|^{\frac{1}{2}}+1)$ and $|A_{6,s}| \le  6^4 \cdot 2^6 \frac{|W| - 1}{|W|^{\frac{1}{s}} - 1}$ by Lemma ~\ref{BdSCRSp}(3) and Lemma ~\ref{counttop2u}, thus we set $a_5 = 6^4 \cdot 2^6 \cdot (|W|^{\frac{1}{2}}+1)$ and $a_{6,s} \le  6^4 \cdot 2^6 \frac{|W| - 1}{|W|^{\frac{1}{s}} - 1}$

    By analogy with the case $e=16$ we obtain the following estimate for summands in $\star$ corresponding to subgroups in $A_6$:
    \[
        \sum_{s \ge 3} a_{6,s} (|W|^{\beta_{6,s}} - 1) \le  \Div(\dim W) \cdot 8b \cdot 6^4 \cdot 2^6 \cdot |W|^{\frac{8b-1}{3}} \cdot (|W|-1)
        .\]

    It is routine to check that $\star$ is satisfied when $|W| \geq 23$ and $|W|$ is a prime or when $|W| \geq 27$ and $|W|$ is a prime power. Also $\star$ is satisfied if $|W|\geq 7$ and $b \geq 2$ or $|W|\ge 4$ and $b \ge 3$ or $|W| \ge 3$ and $b \ge 4$. \\

    Let $e=4$.
    Thus in the notations of Theorem~\ref{Strofprimitiveprime} $|F / U| = |E / Z| = 2^4,\ r=2,\ 2n = 4,\ 2 \mid |W|-1$ and $A/F \in \SCRSp(4,2)$. By Lemma ~\ref{BdSCRSp}(2), it follows that $|A/F| \leq 6^2 \cdot 2$.

    Define $A_i$-s in the following way:
    \begin{enumerate}
        \item $A_1=\{\langle x \rangle\ |\ x \in \EP_2(F \backslash U)\}$,
        \item $A_2=\{\langle x \rangle\ |\ x \in \EP_2(A \backslash F)$, $x$ is a good element and $\bC_{\bar E}(x)=2^2 \}$,
        \item $A_3=\{\langle x \rangle\ |\ x \in \EP_2(A \backslash F)$ and $x$ is bad$\}$,
        \item $A_4=\{\langle x \rangle\ |\ x \in \EP_{\{3,5\}}(A \backslash F)\}$,
        \item $A_5=\{\langle x \rangle\ |\ x \in \EP_2(G \backslash A)\}$,
        \item $A_{6,s} = \left\{ \langle x \rangle  \mid x \in \EP_s(G \backslash A)\right\}$ and $A_6 = \bigcup_{s \ge 3} A_{6,s}$.
    \end{enumerate}

    We set $\beta_i$ as follows:
    \begin{enumerate}
        \item $\beta_1 = 2$ (see Lemma~\ref{roughestimate}(1)),
        \item $\beta_2 = 3$ (see Lemma~\ref{mygood}(2)),
        \item $\beta_3 = 2$ (see Lemma~\ref{mygood}(1)),
        \item $\beta_4 = 2$ (see Lemma~\ref{roughestimate}(3)),
        \item $\beta_5 = 2$ (see Lemma~\ref{roughestimate}(5)),
        \item $\beta_{6,s}$ (see Lemma~\ref{roughestimate}(5)).
    \end{enumerate}

    By analogy to the case $A_1$ for $e=16$, we set $a_1 = 30$.

    We have $|A_2| \le \NPC(A / F, E / Z,2,2) \cdot \max_{xF \in A / F} \NEP_2(xF) \le 15 \cdot 2^2 \cdot 2$ by Lemma ~\ref{BdSCRSp}(2) and Lemma ~\ref{basiccounting}(2), thus we set $a_2 = 15 \cdot 2^2 \cdot 2$
    Let $\langle x \rangle  \in A_2$.
    By analogy with the case $A_2$ ($e=8$), we have decomposition $V = \bC_V(x)\oplus \bC_V(xz)$ for $z \in U, o(z)=2$. Thus $A_2$ can be split into two sets $A_{21}$, $A_{22}$ with corresponding $\beta_{21} = 3, \beta_{22} = 1, a_{21}=a_{22}= \frac{a_2}{2} = 15\cdot 2^2$.

    By analogy to the case $A_4$ for $e=16$, we have $|A_3| \le  21 \cdot 2^4 \cdot 2$ by Lemma ~\ref{BdSCRSp}(2) and Lemma ~\ref{basiccounting}(2) and thus we set $a_3 = 21 \cdot 2^4 \cdot 2$

    By analogy to the case $A_4$ for $e=8$, we have $|A_4| \le 8 \cdot 2^4 \cdot 3 \cdot \frac{1}{2} + 4 \cdot 2^4 \cdot 5 \cdot \frac{1}{4}$ by Lemma ~\ref{BdSCRSp}(2) and Lemma ~\ref{basiccounting}(3), thus we set $a_4 = 17 \cdot 2^4$.

    By analogy to the case $A_6$ for $e=16$, we have $|A_5| \le 6^2 \cdot 2 \cdot 2^4 \cdot (|W|^{\frac{1}{2}}+1)$ and $|A_{6,s}| \le  6^2 \cdot 2 \cdot 2^4 \cdot \frac{|W|-1}{|W|^{\frac{1}{s}} - 1}$ by Lemma ~\ref{BdSCRSp}(2) and Lemma~\ref{counttop2u}. Thus we set $a_5 = 6^2 \cdot 2 \cdot 2^4 \cdot (|W|^{\frac{1}{2}}+1)$ and $a_{6,s} = 6^2 \cdot 2 \cdot 2^4 \cdot \frac{|W|-1}{|W|^{\frac{1}{s}} - 1}$.

    In the same way as in the case for $e=16$, we obtain the following estimate for the summands in $\star$ corresponding to sungroups in $A_6$:
    \[
        \sum_{s \ge 3} a_{6,s} (|W|^{\beta_{6,s}} - 1) \le  \Div(\dim W) \cdot 4b \cdot 6^2 \cdot 2 \cdot  2^4 \cdot |W|^{\frac{4b-1}{3}} \cdot (|W|-1)
        .\]

    It is routine to check that $\star$ is satisfied when $|W| \geq 79$ and $|W|$ is a prime or $|W| \geq 281$ and $|W|$ is a prime power. Also $\star$ is satisfied if $b \geq 2$ and $|W|\geq 13$ or $b \ge 3$ and $|W| \ge 5$ or $b \ge 5$ and $|W| \ge 3$. \\

    Let $e=3$.
    Thus in the notations of Theorem~\ref{Strofprimitiveprime} $|F / U| = |E / Z| = 3^2,\ r = 3, 2n = 2,\ 3 \mid |W|-1$, $A/F \lesssim \SL(2,3)$ and $|A/F| \leq 24$ by Lemma ~\ref{BdSCRSp}(5).

    Define $A_i$ and set $\beta_i$ as follows:
    \begin{enumerate}
        \item $A_1=\{\langle x \rangle\ |\ x \in \EP_3(F \backslash U)\}$, $\beta_1 = 1$ (see Lemma~\ref{roughestimate}(1)),
        \item $A_2=\{\langle x \rangle\ |\ x \in \EP_2(A \backslash F)\}$, $\beta_2 = 2$ (see Lemma~\ref{roughestimate}(2)),
        \item $A_3=\{\langle x \rangle\ |\ x \in \EP_3(A \backslash F)\}$ (below we split $A_3$ into three sets and set betas for them),
        \item $A_4=\{\langle x \rangle\ |\ x \in \EP_2(G \backslash A)\}$, $\beta_4 = 1.5$ (see Lemma~\ref{roughestimate}(5)),
        \item $A_5 = \left\{ \langle x \rangle  \mid x \in \EP_3(G \backslash A) \right\}$, $\beta_5 = 1$ (see Lemma~\ref{roughestimate}(5)),
        \item $A_{6,s} = \left\{ \langle x \rangle  \mid x \in \EP_s(G \backslash A) \right\}$ and $A_6 = \bigcup\limits_{s \ge 5} A_{6,s}$, $\beta_{6,s} = \frac{3}{s}$ (see Lemma~\ref{roughestimate}).
    \end{enumerate}

    By analogy to the case $A_1$ for $e=9$, we have $|A_1| \le  13$ since $\NEP_3(E)\le 26$ and we set $a_1=13$.

    In the same way as in the case $A_2$ for $e=9$, we have $|A_2|\le  1 \cdot 3^2 \cdot 2$ by Lemma ~\ref{BdSCRSp}(5) and Lemma ~\ref{basiccounting}(2) and we set $a_2 = 1 \cdot 3^2 \cdot 2$.

    Let $\langle x \rangle \in A_3$ and $z \in U, |z| = 3$.
    Since $z$ centralizes A, it follows that both $x$, $xz$ and $xz^2$ are good or they both are bad.
    If $x$ is a bad element, then $|\bC_V(x)| \leq |W|^{b}$ by Lemma ~\ref{mygood}(1) and the same holds for $xz$ and $xz^2$.
    Let $x$ be a good element.
    Since $x \not\in F$, it follows that $|\bC_{\bar{E}}(x)| \le 3$ and $|\bC_V(x)| \leq |W|^{2b}$ by Lemma ~\ref{mygood}(2).
    By Lemma~\ref{decomptriv}, we have decomposition $V = \bC_V(x) \oplus \bC_V(xz) \oplus \bC_V(xz^2)$ and the worst case is $|\bC_V(x)| = |W|^{2b}, |\bC_V(x)|=|W|^b$ and $|\bC_V(x)|=1$.
    Thus we split $A_3$  into three sets $A_{31}, A_{32}$ and $A_{33}$ with $\beta_{31} = 2, \beta_{32} = 1, \beta_{33} = 0$ and $a_{31} = a_{32} = a_{33} = 8\cdot \frac{3^2}{2} \cdot \frac{1}{3}$ (see Lemma ~\ref{BdSCRSp}(5) and Lemma ~\ref{basiccounting}(2)).

    By analogy with the case $A_6$ for $e=16$, we have $|A_4| \le  24\cdot 3^2 \cdot (|W|^{\frac{1}{2}}+1)$ by Lemma ~\ref{counttop2u},
    thus we set $a_4 = 24 \cdot 3^2 \cdot (|W|^{\frac{1}{2}}+1)$.

    By analogy to the previous case, we have $|A_5| \le 24 \cdot 3^2 \cdot (|W|^{\frac{2}{3}} + |W|^{\frac{1}{3}} + 1)$
    and we set $a_5 = 24 \cdot 3^2 \cdot (|W|^{\frac{2}{3}} + |W|^{\frac{1}{3}} + 1)$.

    In the same way as in the case $A_7$ for $e=16$, we have $|A_{6,s}| \le 24\cdot 3^2 \cdot  \frac{|W|-1}{|W|^{\frac{1}{s}} - 1}$ by Lemma ~\ref{BdSCRSp}(5), Lemma ~\ref{counttop2u}.
    Thus we set $a_{6,s} = 24\cdot 3^2 \cdot \frac{|W|-1}{|W|^{\frac{1}{s}} - 1}$.
    Obtain the following estimate for corresponding summands in $\star$:
    \[
        \sum_{s \ge 5} a_{6,s} (|W|^{\beta_{6,s}} - 1) \le  \Div(\dim W) \cdot 3b \cdot 24 \cdot 3^2 \cdot |W|^{\frac{3b-1}{5}} \cdot (|W|-1)
        .\]

    It is routine to check that $\star$ is satisfied when $|W| \geq 31$ and $|W|$ is prime or $|W|\geq 439$ and $|W|$ is a prime power. Also $\star$ is satisfied if $|W|\geq 13$ and $b \geq 2$ or $|W|  \ge 7$ and $b\ge 3$ or $|W| \ge 4$ and $b \ge 5$. \\

    Let $e=2$.
    Thus in the notations of Theorem~\ref{Strofprimitiveprime} $|F / U|=|E / Z| = 2^2,\ r =2, 2n = 2,\ 2 \mid |W|-1$, $A/F \lesssim S_3$ and $|A/F| \leq 6$ by Lemma ~\ref{BdSCRSp}(1).

    Define $A_i$ and set $\beta_i$ as follows:
    \begin{enumerate}
        \item $A_1=\{\langle x \rangle\ |\ x \in \EP_2(F \backslash U)\}$, $\beta_1=1$ (see Lemma ~\ref{roughestimate}(1)),
        \item $A_2=\{\langle x \rangle\ |\ x \in \EP_2(A \backslash F)\}$, $\beta_2 = 1$ (see Lemma~\ref{mygood}(1), Lemma~\ref{mygood}(2) and the case $A_3$ for $e=3$),
        \item $A_3=\{\langle x \rangle\ |\ x \in \EP_3(A \backslash F)\}$, $\beta_3 = 1$ (see Lemma ~\ref{roughestimate}(3)),
        \item $A_4=\{\langle x \rangle\ |\ x \in \EP_2(G \backslash A)\}$, $\beta_4 = 1$ (see Lemma ~\ref{roughestimate}(5)),
        \item $A_5 = \left\{ \langle x \rangle  \mid x \in \EP_3(G \backslash A) \right\} $, $\beta_5 = \frac{2}{3}$ (see Lemma ~\ref{roughestimate}(5)),
        \item $A_{6,s} = \{ \langle x \rangle   \mid x \in \EP_s(G \backslash A)\}$ and $A_6 = \bigcup\limits_{s \ge 5} A_{6,s}$, $\beta_{6,s} = \frac{2}{s}$ (see Lemma ~\ref{roughestimate}(5)).
    \end{enumerate}

    By analogy to the cases $A_1$ for other even $e$, we have $|A_1| \le 6$ since $\NEP_2(E) \le 7$ and we set $a_1 = 6$.

    By analogy to the cases $A_2,\ A_4$ for $e = 16$, we have $|A_2| \le 3\cdot 4$ by Lemma ~\ref{BdSCRSp}(1) and Lemma ~\ref{basiccounting}(2) and we set $a_2=3\cdot 4$.

    In the same way as in the case $A_5$ for $e=16$, we have $|A_3| \le 2\cdot 4\cdot \frac{3}{2}$ by Lemma ~\ref{BdSCRSp}(1) and Lemma ~\ref{basiccounting}(2) and we set $a_3 = 2\cdot 4\cdot \frac{3}{2}$.

    We have $|A_4| \le 6 \cdot 2^2 \cdot (|W|^{\frac{1}{2}} + 1)$ by Lemma ~\ref{counttop2u},
    thus we set $ a_4=6 \cdot 2^2 \cdot (|W|^{\frac{1}{2}} + 1)$.

    As in the previous case, we have $|A_5| \le 6 \cdot 2^2 \cdot (|W|^\frac{2}{3} + |W|^\frac{1}{3} + 1)$,
    thus we set $a_5 = 6 \cdot 2^2 \cdot (|W|^\frac{2}{3} + |W|^\frac{1}{3} + 1)$.

    In the same way as in the case $A_6$ for $e=3$, we have $|A_{6,s}| \le 6 \cdot 2^2 \cdot \frac{|W|-1}{|W|^{\frac{1}{s}} - 1}$ and we set $a_{6,s}  = 6 \cdot 2^2 \cdot  \frac{|W|-1}{|W|^{\frac{1}{s}} - 1}$. We obtain the following estimate:
    \[
        \sum_{s \ge 3} a_{6,s} (|W|^{\beta_{6,s}} - 1) \le  \Div(\dim W) \cdot 2b \cdot 6 \cdot 2^2 \cdot  |W|^{\frac{2b-1}{5}} \cdot (|W|-1)
        .\]

    It is routine to check that $\star$ is satisfied when $|W| \geq 31$ and $|W|$ is a prime, or when $|W| \geq 5693$ and $|W|$ is a prime power. Also if $b \geq 2\ (3,4$ or $5)$, $\star$ is satisfied when $|W| \geq 19\ (7,5$ or $3$ respectively$)$.

\end{proof}

\begin{tab}{List of primitive solvable groups which possibly do not have regular orbit}\label{exception}
    \centering
    \begin{tabular}{|c|c|c|c|c|}
        \hline
        No. & $e$ & $p$ & $d$ & $a$ \\
        \hline
        1   & 16  & 3   & 16  & 1   \\
        2   & 16  & 5   & 16  & 1   \\ \hline
        3   & 9   & 2   & 18  & 2   \\
        4   & 9   & 7   & 9   & 1   \\
        5   & 9   & 13  & 9   & 1   \\
        6   & 9   & 2   & 36  & 4   \\
        7   & 9   & 19  & 9   & 1   \\
        8   & 9   & 5   & 18  & 2   \\ \hline
        9   & 8   & 3   & 8   & 1   \\
        10  & 8   & 5   & 8   & 1   \\
        11  & 8   & 7   & 8   & 1   \\
        12  & 8   & 3   & 16  & 2   \\
        13  & 8   & 11  & 8   & 1   \\
        14  & 8   & 13  & 8   & 1   \\
        15  & 8   & 17  & 8   & 1   \\
        16  & 8   & 19  & 8   & 1   \\
        17  & 8   & 5   & 16  & 2   \\
        18  & 8   & 3   & 24  & 3   \\ \hline
        19  & 4   & 3   & 4   & 1   \\
        20  & 4   & 5   & 4   & 1   \\
        21  & 4   & 7   & 4   & 1   \\
        22  & 4   & 3   & 8   & 2   \\
        23  & 4   & 11  & 4   & 1   \\
        24  & 4   & 13  & 4   & 1   \\
        25  & 4   & 17  & 4   & 1   \\
        26  & 4   & 19  & 4   & 1   \\
        \hline
    \end{tabular}
    \quad
    \begin{tabular}{|c|c|c|c|c|}
        \hline
        No. & $e$ & $p$ & $d$ & $a$ \\
        \hline
        27  & 4   & 23  & 4   & 1   \\
        28  & 4   & 5   & 8   & 2   \\
        29  & 4   & 3   & 12  & 3   \\
        30  & 4   & 29  & 4   & 1   \\
        31  & 4   & 31  & 4   & 1   \\
        32  & 4   & 37  & 4   & 1   \\
        33  & 4   & 41  & 4   & 1   \\
        34  & 4   & 43  & 4   & 1   \\
        35  & 4   & 47  & 4   & 1   \\
        36  & 4   & 7   & 8   & 2   \\
        37  & 4   & 53  & 4   & 1   \\
        38  & 4   & 59  & 4   & 1   \\
        39  & 4   & 61  & 4   & 1   \\
        40  & 4   & 67  & 4   & 1   \\
        41  & 4   & 71  & 4   & 1   \\
        42  & 4   & 73  & 4   & 1   \\
        43  & 4   & 3   & 16  & 4   \\
        44  & 4   & 11  & 8   & 2   \\
        45  & 4   & 5   & 12  & 3   \\
        46  & 4   & 13  & 8   & 2   \\
        47  & 4   & 3   & 20  & 5   \\ \hline
        48  & 3   & 2   & 6   & 2   \\
        49  & 3   & 7   & 3   & 1   \\
        50  & 3   & 13  & 3   & 1   \\
        51  & 3   & 2   & 12  & 4   \\
        52  & 3   & 19  & 3   & 1   \\
        \hline
    \end{tabular}
    \quad
    \begin{tabular}{|c|c|c|c|c|}
        \hline
        No. & $e$ & $p$ & $d$ & $a$ \\
        \hline
        53  & 3   & 5   & 6   & 2   \\
        54  & 3   & 7   & 6   & 2   \\
        55  & 3   & 2   & 18  & 6   \\
        56  & 3   & 11  & 6   & 2   \\
        57  & 3   & 13  & 6   & 2   \\
        58  & 3   & 2   & 24  & 8   \\
        59  & 3   & 17  & 6   & 2   \\
        60  & 3   & 7   & 9   & 3   \\
        61  & 3   & 19  & 6   & 2   \\ \hline
        62  & 2   & 3   & 2   & 1   \\
        63  & 2   & 5   & 2   & 1   \\
        64  & 2   & 7   & 2   & 1   \\
        65  & 2   & 3   & 4   & 2   \\
        66  & 2   & 11  & 2   & 1   \\
        67  & 2   & 13  & 2   & 1   \\
        68  & 2   & 17  & 2   & 1   \\
        69  & 2   & 19  & 2   & 1   \\
        70  & 2   & 23  & 2   & 1   \\
        71  & 2   & 5   & 4   & 2   \\
        72  & 2   & 3   & 6   & 3   \\
        73  & 2   & 29  & 2   & 1   \\
        74  & 2   & 7   & 4   & 2   \\
        75  & 2   & 3   & 8   & 4   \\
        76  & 2   & 11  & 4   & 2   \\
        77  & 2   & 5   & 6   & 3   \\
        78  & 2   & 13  & 4   & 2   \\
        \hline
    \end{tabular}
    \quad
    \begin{tabular}{|c|c|c|c|c|}
        \hline
        No. & $e$ & $p$ & $d$ & $a$ \\
        \hline
        79  & 2   & 3   & 10  & 5   \\
        80  & 2   & 17  & 4   & 2   \\
        81  & 2   & 7   & 6   & 3   \\
        82  & 2   & 19  & 4   & 2   \\
        83  & 2   & 23  & 4   & 2   \\
        84  & 2   & 5   & 8   & 4   \\
        85  & 2   & 3   & 12  & 6   \\
        86  & 2   & 29  & 4   & 2   \\
        87  & 2   & 31  & 4   & 2   \\
        88  & 2   & 11  & 6   & 3   \\
        89  & 2   & 37  & 4   & 2   \\
        90  & 2   & 41  & 4   & 2   \\
        91  & 2   & 43  & 4   & 2   \\
        92  & 2   & 3   & 14  & 7   \\
        93  & 2   & 13  & 6   & 3   \\
        94  & 2   & 47  & 4   & 2   \\
        95  & 2   & 7   & 8   & 4   \\
        96  & 2   & 53  & 4   & 2   \\
        97  & 2   & 5   & 10  & 5   \\
        98  & 2   & 59  & 4   & 2   \\
        99  & 2   & 61  & 4   & 2   \\
        100 & 2   & 67  & 4   & 2   \\
        101 & 2   & 17  & 6   & 3   \\
        102 & 2   & 71  & 4   & 2   \\
        103 & 2   & 73  & 4   & 2   \\
        \hline
    \end{tabular}
\end{tab}

\begin{cor}\label{except}
    Assume that $G$ is a solvable irreducible primitive subgroup of $\GL(d,p)$. Suppose that the parameters of $G$ are $e, a$ in the notations of Theorem~{\em \ref{Strofprimitiveprime}}, where $e \in \{2,3,4,8,9,16\}$ and $e\cdot a=d$. Suppose that $G$ is not mentioned in Table~{\em \ref{exception}}. Then $G$ has a regular orbit on $V = \FF_p^d$.
\end{cor}

\section{Acknowledgement} \label{sec:Acknowledgement}
The research of the first author  was partially supported by the NSFC (No: 11671063) and a grant from the Simons Foundation (No 499532). The second and the third author are supported by the Russian Scientific Foundation (project No 19-11-00039).

\small

\end{document}